\documentclass[12pt, reqno]{amsart}
\usepackage{amsmath, amsthm, amscd, amsfonts, amssymb, graphicx, color}
\usepackage[bookmarksnumbered, colorlinks, plainpages]{hyperref}
\hypersetup{colorlinks=true,linkcolor=red, anchorcolor=green, citecolor=cyan, urlcolor=red, filecolor=magenta, pdftoolbar=true}

\usepackage{tensor}
\usepackage{mathrsfs}
\usepackage{cleveref}

\newtheorem{theorem}{Theorem}[section]
\newtheorem{lemma}[theorem]{Lemma}

\newtheorem{corollary}[theorem]{Corollary}
\theoremstyle{definition}
\newtheorem{definition}[theorem]{Definition}
\newtheorem{example}[theorem]{Example}

\theoremstyle{remark}
\newtheorem{remark}[theorem]{Remark}
\numberwithin{equation}{section}

\usepackage{fullpage}

\begin{document}

\setcounter{page}{1}

\title[MPP on QPMs]{Mix-point property in quasi-pseudometric spaces. }

\author[Ya\'e Olatoundji Gaba]{Ya\'e Olatoundji Gaba$^{1,*}$}

\address{$^{1}$Department of Mathematics and Applied Mathematics, University of Cape Town, South Africa.}
\email{\textcolor[rgb]{0.00,0.00,0.84}{gabayae2@gmail.com
}}




\begin{abstract}

In this article, we give new results in the startpoint theory for quasi-pseudometric spaces.
The results we present provide us with the existence of startpoint (endpoint, fixed point) for multi-valued maps defined on a bicomplete quasi-pseudometric space. We characterise the existence of startpoint and endpoint by the so-called \textit{mix-point property}. The present results extend known ones in the area.

\end{abstract} 

\maketitle

\section{Introduction and preliminaries}
The theory of startpoint, first introduced in \cite{rico}, came to extend the idea of fixed points for multi-valued mappings defined on quasi-pseudometric spaces. A series of three papers, see \cite{rico, rico1, ricoo} has given a more or less detailed introduction to the subject. The aim of the present article is to continue this study by introducing the idea of \textit{mix-point property}, which is used to characterise the existence of startpoints.

\begin{definition}
Let $X$ be a non empty set. A function $d:X \times X \to [0,\infty)$ is called a \textbf{quasi-pseudometric} on $X$ if:
\begin{enumerate}
\item[i)] $d(x,x)=0 \quad \forall \ x \in X$, 
\item[ii)] $d(x,z) \leq d(x,y) + d(y,z) \quad \forall\  x,y,z \in X $. 
\end{enumerate}

Moreover, if

\begin{enumerate}
\item[iii)]  $d(x,y)=0=d(y,x) \Longrightarrow x=y$, then $d$ is said to be a \textbf{$T_0$-quasi-metric}.
\end{enumerate}
 The latter condition is referred to as the $T_0$-condition.
\end{definition}

\begin{remark} \hspace*{0.5cm}  
\begin{itemize}
\item Let $d$ be a quasi-pseudometric on $X$, then the function $d^{-1}$ defined by $d^{-1}(x,y)=d(y,x)$ whenever $x,y \in X$ is also a quasi-pseudometric on $X$, called the \textbf{conjugate} of $d$. In the literature, $d^{-1}$ is also denoted $d^t$ or $\bar{d}$.
\item It is easy to verify that the function $d^s$ defined by $d^s:=d\vee d^{-1}$, i.e. $d^s(x,y)=\max \{d(x,y),d(y,x)\}$ defines a metric on $X$ whenever $d$ is a $T_0$-quasi-metric on $X$.
\end{itemize}

\end{remark}

\begin{definition}\cite{rico}
 A $T_0$-quasi-metric space $(X,d)$ is called \textbf{bicomplete} provided that the metric $d^s$ on $X$ is complete.
\end{definition}

Let $(X,d)$ be a quasi-pseudometric space. We set $\mathscr{P}_0(X):=2^X \setminus \{ \emptyset\}$ where $2^X$ denotes the power set of $X$. For $x\in X$ and $A \in \mathscr{P}_0(X)$, we define:

$$ d(x,A)= \inf\{ d(x,a),a\in A\} , \quad  d(A,x)= \inf\{ d(a,x),a\in A\}.$$

We also define the map $H:\mathscr{P}_0(X) \times \mathscr{P}_0(X) \to [0,\infty]$ by 

$$H(A,B)= \max \left\lbrace \underset{a\in A}{\sup}\ d(a,B), \ \underset{b\in B}{\sup} \ d(A,b)   \right\rbrace \text{ whenever } A,B, \in \mathscr{P}_0(X).$$

Then $H$ is an extended\footnote{This means that $H$ can attain the value $\infty$ as it appears in the definition.} quasi-pseudometric on $\mathscr{P}_0(X)$.

\section{Some first results}
We briefly recall the idea of a startpoint, as initially intended in \cite{rico}.

\begin{definition}(Compare \cite{rico})
Let $(X,d)$ be a $T_0$-quasi-metric space.

Let $F:X\to 2^X$ be a set-valued map. An element $x\in X$ is said to be 
\begin{enumerate}
\item[(i)] a fixed point of $F$ if $x\in Fx$,
\item[(ii)] a startpoint of $F$ if $H(\{x\},Fx)=0$,
\item[(iii)] an endpoint of $F$ if $H(Fx,\{x\})=0$.
\item[(iv)]  an $\varepsilon$-startpoint of $F$ for some $\varepsilon\in (0,1)$ if  $H(\{x\},Fx)< \varepsilon$.
\item[(v)]  an $\varepsilon$-endpoint of $F$ for some $\varepsilon\in (0,1)$ if  $H(Fx,\{x\})< \varepsilon$.
\end{enumerate}
\end{definition}

\begin{remark}( See \cite{rico})
It is therefore obvious that if $x$ is both a startpoint of $F$ and an endpoint of $F$, then $x$ is a fixed point of $F$. In fact, $Fx$ is a singleton. Observe that a fixed point need not to be a startpoint nor an endpoint. 

Indeed, consider the $T_0$-quasi-metric space $(X,d)$ where $X=\{0,1\}$ and $d$ defined by $d(0,1)=0, \ d(1,0)=1$ and $d(x,x)=0$ for $x=0,1$. We define on $X$ the set-valued map $F:X\to 2^X$ by $Fx=X$. Obviously, $1$ is a fixed point, but $H(\{1\},F1)=H(\{1\},X)=\max\{d(1,1),d(1,0)\}=1\neq 0$.
\end{remark}

\begin{lemma}\cite{rico}
Let $(X,d)$ be a $T_0$-quasi-metric space and $F:X\to 2^X$ be a set-valued map. An element $x\in X$ is a startpoint of $F$ if and only if it is an $\varepsilon$-startpoint of $F$ for every $\varepsilon\in (0,1)$. 
\end{lemma}

\begin{lemma}\cite{rico}
Let $(X,d)$ be a $T_0$-quasi-metric space and $F:X\to 2^X$ be a set-valued map. An element $x\in X$ is an endpoint of $F$ if and only if it is an $\varepsilon$-endpoint of $F$ for every $\varepsilon\in (0,1)$. 
\end{lemma}

\begin{definition}\cite{rico}
Let $(X,d)$ be a $T_0$-quasi-metric space. We say that a set-valued map $F:X\to 2^X$ has the {\bf approximate startpoint property} (resp. {\bf approximate endpoint property} ), if
$$\underset{x\in X}{\inf}\underset{y\in Fx}{\sup} d(x,y)=0\ \ (\text{resp.}\  \underset{x\in X}{\inf}\underset{y\in Fx}{\sup} d(y,x)=0).$$
\end{definition}

\begin{definition}\cite{rico}
Let $(X,d)$ be a $T_0$-quasi-metric space. We say that a set-valued map $F:X\to 2^X$ has the {\bf approximate mix-point property} if
$$\underset{x\in X}{\inf}\underset{y\in Fx}{\sup} d^s(x,y)=0.$$
\end{definition}
Here, it is also very clear that $F$ has the approximate mix-point property if and only if $F$ has both 
the approximate startpoint and the approximate endpoint properties.

We illustrate our definitions by giving the following example. 
\begin{example}(Compare \cite{rico})
Let $X=\{0,1,2 \}$. The map $d:X\times X \to [0,\infty)$ defined by $d(0,1)=d(0,2)=0,\ d(1,0)=d(1,2)=1$, $d(2,0)=d(2,1)=2$ and $d(x,x)=0$ for all $x\in X$ is a $T_0$-quasi-pseudometric on $X$.
Let $F:X\to 2^X$ be the set mapping defined by $Fa=X\setminus \{a\}$ for any $a\in X$.
Then $F$ admits $0$ as unique startpoint, no endpoint and no fixed point. Moreover, $F$  has the approximate startpoint property but does not have the approximate endpoint property.
\end{example}

We recall below the main theorem in the startpoint theory that appeared in \cite{rico}.

\begin{theorem}\cite[Theorem 29]{rico}\label{thm1}
Let $(X,d)$ be a bicomplete quasi-pseudometric space. Let $F:X\to CB(X)$ be a set-valued map that satisfies
\begin{equation}
H(Fx,Fy) \leq \psi (d(x,y)), \ \text{ \ for each } x,y\in X,
\end{equation}
where $\psi:[0,\infty)\to [0,\infty)$ is upper semicontinuous, $\psi(t)<t$ for each $t>0$ and $\underset{t \to \infty }{\liminf} (t-\psi(t))>0$. Then there exists a unique $x_0\in X$ which is both a startpoint and an endpoint  of $F$ if and only if $F$ has the approximate mix-point property.
\end{theorem}

We introduce the following definitions: 

\begin{definition} Let $(X,d)$ be a  quasi-pseudometric space, $J :X\to X$ be a single valued mapping and $F:X\to 2^X$ be a multi-valued mapping. We say that the mappings $J$ and $F$ have the {\bf approximate startpoint property} (resp. {\bf approximate endpoint property} ), if
$$\underset{x\in X}{\inf}\underset{y\in Fx}{\sup} d(Jx,y)=0\ \ (\text{resp.}\  \underset{x\in X}{\inf}\underset{y\in Fx}{\sup} d(y,Jx)=0).$$
\end{definition}

\begin{definition}
Let $(X,d)$ be a $T_0$-quasi-pseudometric space, $J :X\to X$ be a single valued mapping. We say that $J$ and the set-valued map $F:X\to 2^X$ have the {\bf approximate mix-point property} if
$$\underset{x\in X}{\inf}\underset{y\in Fx}{\sup} d^s(Jx,y)=0.$$
\end{definition}

\begin{definition}(Compare \cite{rico})
Let $(X,d)$ be a  quasi-pseudometric space, $J :X\to X$ be a single valued mapping. Let $F:X\to 2^X$ be a set-valued map. An element $x\in X$ is said to be 
\begin{enumerate}
\item[(i)] a $J$-fixed point of $F$ if $Jx\in Fx$,
\item[(ii)] a startpoint of $J$ and $F$ if $H(\{Jx\},Fx)=0$,
\item[(iii)] an endpoint of $J$ and $F$ if $H(Fx,\{Jx\})=0$.
\end{enumerate}
\end{definition}

The next three results are the first results of this paper. We shall not give any proof, since the proofs follow the same arguments as the proofs in \cite{rico}.

\begin{theorem}\label{Res}(Compare\cite[Theorem 29]{rico})
Let $(X,d)$ be a bicomplete quasi-pseudometric space. Assume $J:X\to X$ is a continuous single-valued map and let $F:X\to CB(X)$ be a set-valued map that satisfiy
\begin{equation}
H(Fx,Fy) \leq \psi (d(Jx,Jy)), \ \text{ \ for each } x,y\in X,
\end{equation}
where $\psi:[0,\infty)\to [0,\infty)$ is usc\footnote{for upper semicontinuous}, $\psi(t)<t$ for each $t>0$ and $\underset{t \to \infty }{\liminf} (t-\psi(t))>0$. Then there exists a unique $x_0\in X$ which is both a startpoint and an endpoint  of $J$ and $F$ if and only if $J$ and $F$ have the approximate mix-point property.
\end{theorem}

\begin{theorem}\label{Resone}(Compare\cite[Theorem 31]{rico})
Let $(X,d)$ be a bicomplete quasi-pseudometric space. Assume $J:X\to X$ is a continuous single-valued map and let $F:X\to CB(X)$ be a set-valued map that satisfy
\begin{equation}
H(Fx,Fy) \leq k (d(Jx,Jy)), \ \text{ \ for each } x,y\in X,
\end{equation}
where $k\in [0,1)$. Then there exists a unique $x_0\in X$ which is both a startpoint and an endpoint  of $J$ and $F$ if and only if $J$ and $F$ have the approximate mix-point property.
\end{theorem}

\begin{theorem}\label{Restwo}(Compare\cite[Corollary 30]{rico})
Let $(X,d)$ be a bicomplete quasi-pseudometric space. Assume $J:X\to X$ is a continuous single-valued map. Let $F:X\to CB(X)$ be a set-valued map that satisfies
\begin{equation}
H(Fx,Fy) \leq \psi (d(Jx,Jy)), \ \text{ \ for each } x,y\in X,
\end{equation}
where $\psi:[0,\infty)\to [0,\infty)$ is an upper semicontinuous map that satisfies $\psi(t)<t$ for each $t>0$ and $\underset{t \to \infty }{\liminf} (t-\psi(t))>0$. If $J$ and $F$ have the approximate mix-point property then $F$ has a $J$-fixed point.
\end{theorem}

\begin{remark}
Observe that if we put $J=I_X$(identity map on $X$) in Theorems \ref{Res}, \ref{Resone} and \ref{Restwo} respectively, we obtain \cite[Theorem 29, Theorem 31, Corollary 30]{rico} respectively.
\end{remark}

\section{More results}
In \cite{rico}, the proof of Theorem \ref{thm1} basically establishes that the sets 

\[
 C_n=\left\lbrace x \in X: \underset{y\in Fx}{\sup} d^s(x,y)\leq \frac{1}{n} \right\rbrace \neq\emptyset, \ \quad \text{ for } n\in \mathbb{N}=\{1,2,\cdots\},
\]
form a non-increasing sequence of bounded and $\tau(d^s)$-closed sets. The conclusion follows from the Cantor intersection theorem. We shall use a similar approach in proving the next two results, with the difference that we present simpler and shorter arguments. 

We now present the first non trivial generalisation of \cite[Theorem 29]{rico}.

\begin{theorem}\label{Res1}
Let $(X,d)$ be a bicomplete quasi-pseudometric space. Assume $J:X\to X$ is a continuous single-valued map such that $rd(x,y)\leq d(Jx,Jy)$ for some constant $r>0$. Let $F:X\to CB(X)$ be a set-valued map that satisfies
\begin{equation}\label{res1-eq}
H(Fx,Fy) \leq \alpha d(Jx,Jy), \ \text{ \ for each } x,y\in X,
\end{equation}
where $\alpha\in (0,1)$ and $r\alpha<1$. Then there exists a unique $x_0\in X$ which is both a startpoint and an endpoint of $J$ and $F$ if and only if $J$ and $F$ have the approximate mix-point property.
\end{theorem}

\begin{proof}

It is clear that if  $J$ and $F$ admit a point which is both a startpoint and an endpoint, then  $J$ and $F$ have the approximate startpoint property and the approximate endpoint property, i.e the approximate mix-point property.
Conversely, suppose $J$ and $F$ have the approximate mix-point property. Then

\[
C_n= \left\lbrace x \in X: \underset{y\in Fx}{\sup} d^s(Jx,y)\leq \frac{1}{n} \right\rbrace \neq\emptyset,
\]
for each $n\in \mathbb{N}$. Moreover for each $n\in \mathbb{N}$, $C_{n+1} \subseteq C_n$. Since the map $x\mapsto \underset{y\in Fx}{\sup} d^s(Jx,y)$ is $\tau(d^s)$-lsc\footnote{For lower semicontinuous} (as supremum of $\tau(d^s)$-continuous mappings), then $C_n$ is $\tau(d^s)$-closed.

Next we prove that for each $n\in \mathbb{N},$ $C_n$ is bounded. Indeed, for any $x,y \in C_n,$

\begin{align*}
d(Jx,Jy) & = H(\{Jx\},\{Jy\})\\
         & \leq H(\{Jx\},Fx) + H(Fx,Fy) + H(Fy,\{Jy\})\\
         & \leq \frac{2}{n} + \alpha d(Jx,Jy).
\end{align*}
So
 \[
 d(Jx,Jy) \leq \frac{2}{n(1-\alpha)}, 
\]

and since $rd(x,y)\leq d(Jx,Jy)$, we have 

\[
\delta(C_n) \leq \frac{2}{rn(1-\alpha)} .
\]
Therefore $\underset{n \rightarrow \infty}{\lim} \delta(C_n) = 0$. It follows from the Cantor intersection theorem that $$\underset{n \in N}{\bigcap}C_n = \{x_0\}.$$ 

Thus $$H(\{Jx_0\},Fx_0)=\underset{y\in Fx_0}{\sup} d(Jx_0,y)=0=\underset{y\in Fx_0}{\sup} d(y,Jx_0)=H(Fx_0,\{Jx_0\}).$$  For uniqueness, if $z_0$ is an arbitrary startpoint and endpoint of $J$ and $F$, then $$H(\{Jz_0\},Fz_0)=0=H(Fz_0,\{Jz_0\}),$$ and so $z_0 \in \underset{n \in N}{\bigcap}C_n = \{x_0\}$.
\end{proof}

\vspace*{0.2cm}

We give the following example to illustrate our result.

\begin{example}

Consider the $T_0$-quasi-metric space $(X,d)$ where $X=\{0,1\}$ and $d$ be defined by $d(0,1)=0, \ d(1,0)=1$ and $d(x,x)=0$ for $x=0,1$. Note that $(X,d)$ is bicomplete. We define on $X$ the set-valued map $F:X\to 2^X$ by $Fx=\{0\}$ and the single-valued continuous mapping $J:X\to X$ by $Jx=x^2$. It is clear that for all $x,y \in X,H(Fx,Fy)=0$.

For $x=0, y=1, \ d(x,y)=d(0,1)= 0,$ and $ d(Jx,Jy)=d(0,1)= 0.$ 

For $x=1, y=0, \ d(x,y)=d(1,0)= 1,$ and $ d(Jx,Jy)=d(1,0)= 1.$ 

So if we set $r=\frac{1}{2}$ and $\alpha=\frac{1}{3}$, we have that $r\alpha = \frac{1}{6}<1$ and

$$rd(x,y)\leq d(Jx,Jy).$$

Moreover, the condition \eqref{res1-eq} is satisfied. Since $H(\{J0\},F0)=0=H(F0,\{J0\})$, then $0$ is both a startpoint and an endpoint of $J$ and $F$.

\newpage

Observe also that:
 
-for $x=0, Fx=\{0\}, \underset{y\in Fx}{\sup} d^s(Jx,y)= d^s(0,0)=0,$

-for $x=1, Fx=\{0\}, \underset{y\in Fx}{\sup} d^s(Jx,y)=d^s(1,0)=1,$
and hence
$$\underset{x\in X}{\inf}\underset{y\in Fx}{\sup} d^s(Jx,y)=0.$$
i.e. $J$ and $F$ have the approximate mix-point property.
\end{example}

\begin{corollary}
Let $(X,d)$ be a bicomplete quasi-pseudometric space. Assume $J:X\to X$ is a continuous single-valued map such that $rd(x,y)\leq d(Jx,Jy)$ for some constant $r>0$ and $ \text{ for each } x,y\in X$. Let $F:X\to CB(X)$ be a set-valued map that satisfies
\begin{equation}
H(Fx,Fy) \leq \alpha d(Jx,Jy), \ \text{ \ for each } x,y\in X,
\end{equation}
where $\alpha\in (0,1)$ and $r\alpha<1$. If $J$ and $F$ have the approximate mix-point property then $F$ has a $J$-fixed point.
\end{corollary}

\begin{proof}

From Theorem \ref{Res1}, we conclude that there exists $x_0 \in X$ which is both a startpoint and an endpoint for $J$ and $F$, i.e $H(\{Jx_0\},Fx_0)=0=H(Fx_0,\{Jx_0\})$. The $T_0$-condition therefore guarantees the desired result.

\end{proof}

\begin{theorem}\label{Res2}
Let $(X,d)$ be a bicomplete quasi-pseudometric space. Assume $J:X\to X$ is a continuous single-valued map such that $rd(x,y)\leq d(Jx,Jy)$ for some constant $r>0$ and $ \text{ for each } x,y\in X$. Let $F:X\to CB(X)$ be a set-valued map that satisfies
\begin{equation}\label{res2-eq}
H(Fx,Fy) \leq \alpha[ d(Jx,Fx ) + d(Jy,Fy)], \ \text{ \ for each } x,y\in X,
\end{equation}
where $\alpha\in (0,1/2)$. Then there exists a unique $x_0\in X$ which is both a startpoint and an endpoint of $J$ and $F$ if and only if $J$ and $F$ have the approximate mix-point property.
\end{theorem}

\begin{proof}
Once again, only one implication will be of interest to us, since the other one is trivial. So suppose  $J$ and $F$ have the approximate mix-point property. Then we already know that the sets 

\[
C_n= \left\lbrace x \in X: \underset{y\in Fx}{\sup} d^s(Jx,y)\leq \frac{1}{n} \right\rbrace \neq\emptyset,
\]
for each $n\in \mathbb{N}$ are $\tau(d^s)$-closed and that $C_{n+1} \subseteq C_n$.

Next we prove that for each $n\in \mathbb{N},$ $C_n$ is bounded. Indeed, for any $x,y \in C_n,$

\begin{align*}
d(Jx,Jy) & = H(\{Jx\},\{Jy\})\\
         & \leq H(\{Jx\},Fx) + H(Fx,Fy) + H(Fy,\{Jy\})\\
         & \leq \frac{2}{n} + \alpha [d(Jx,Fx ) + d(Jy,Fy)]\\
         & \leq \frac{1}{n} (2+2\alpha).
\end{align*}
So
 \[
 d(Jx,Jy) \leq \frac{1}{n} (2+2\alpha),
\]

and since $rd(x,y)\leq d(Jx,Jy)$, we have 

\[
\delta(C_n) \leq \frac{1}{rn} (2+2\alpha).
\]
Therefore $\underset{n \rightarrow \infty}{\lim} \delta(C_n) = 0$. It follows from the Cantor intersection theorem that $$\underset{n \in N}{\bigcap}C_n = \{x_0\}.$$ Thus $$H(\{Jx_0\},Fx_0)=\underset{y\in Fx_0}{\sup} d(Jx_0,y)=0=\underset{y\in Fx_0}{\sup} d(y,Jx_0)=H(Fx_0,\{Jx_0\}).$$  For uniqueness, if $z_0$ is an arbitrary startpoint and endpoint of $J$ and $F$, then $$H(\{Jz_0\},Fz_0)=0=H(Fz_0,\{Jz_0\}),$$ and so $z_0 \in \underset{n \in N}{\bigcap}C_n = \{x_0\}$.

\end{proof}

\begin{example}
Consider the $T_0$-quasi-metric space $(X,d)$ where $X=\{0,1\}$ and $d$ defined by $d(0,1)=0, \ d(1,0)=1$ and $d(x,x)=0$ for $x=0,1$. Note that $(X,d)$ is bicomplete. We define on $X$ the set-valued map $F:X\to 2^X$ by $Fx=\{0\}$ and the single valued continuous map $J:X\to X$ by $Jx=x^3$. It is clear that for all $x,y \in X,H(Fx,Fy)=0$.

For $x=0, y=1, d(x,y)=d(0,1)= 0, d(Jx,Jy)=d(0,1)= 0. $ 

For $x=1, y=0, d(x,y)=d(1,0)= 1, d(Jx,Jy)=d(1,0)= 1. $ 
So if we set $r=\frac{1}{2}$, we have that 

$$rd(x,y)\leq d(Jx,Jy).$$

So if we set $\alpha=\frac{1}{3}$, we have that $0<\alpha<\frac{1}{2}$. Since $H(\{J0\},F0)=0=H(F0,\{J0\})$, then $0$ is both a startpoint and an endpoint of $J$ and $F$.

For $x=0, y=1,\alpha[ d(Jx,Fx ) + d(Jy,Fy)] = \frac{1}{3}$ and 

for $x=1, y=0,\alpha[ d(Jx,Fx ) + d(Jy,Fy)] = \frac{1}{3}$, so the condition \eqref{res2-eq} is satisfied.

Observe also that 
 
$$\underset{x\in X}{\inf}\underset{y\in Fx}{\sup} d^s(Jx,y)=0.$$
i.e. $J$ and $F$ have the approximate mix-point property.
\end{example}

\begin{corollary}
Let $(X,d)$ be a bicomplete quasi-pseudometric space. Assume $J:X\to X$ is a continuous single-valued map such that $rd(x,y)\leq d(Jx,Jy)$ for some constant $r>0$ and $ \text{ for each } x,y\in X$. Let $F:X\to CB(X)$ be a set-valued map that satisfies

\begin{equation}
H(Fx,Fy) \leq \alpha[ d(Jx,Fx ) + d(Jy,Fy)], \ \text{ \ for each } x,y\in X,
\end{equation}
where $\alpha\in (0,1/2)$. If $J$ and $F$ have the approximate mix-point property then $F$ has a $J$-fixed point.
\end{corollary}

\newpage

Using the same idea as in the proof of Theorem \ref{Res1} and Theorem \ref{Res2}, one can establish the following results:

\begin{theorem}\label{Res3}
Let $(X,d)$ be a bicomplete quasi-pseudometric space. Assume $J:X\to X$ is a continuous single-valued map such that $rd(x,y)\leq d(Jx,Jy)$ for some constant $r>0$ and $ \text{ for each } x,y\in X$. Let $F:X\to CB(X)$ be a set-valued map that satisfies
\begin{equation}
H(Fx,Fy) \leq \alpha[ d(Jx,Fy ) + d(Fx,Jy)], \ \text{ \ for each } x,y\in X,
\end{equation}
where $0<\alpha<1/2$ with $2r\alpha<1$. Then there exists a unique $x_0\in X$ which is both a startpoint and an endpoint of $J$ and $F$ if and only if $J$ and $F$ have the approximate mix-point property.
\end{theorem}

\begin{theorem}\label{Res4}
Let $(X,d)$ be a bicomplete quasi-pseudometric space. Assume $J:X\to X$ is a continuous single-valued map such that $rd(x,y)\leq d(Jx,Jy)$ for some constant $r>0$ and $ \text{ for each } x,y\in X$.  Let $F:X\to CB(X)$ be a set-valued map that satisfies
\begin{equation}
H(Fx,Fy) \leq \alpha d(Jx,Jy ) + L d(Fx,Jy), \ \text{ \ for each } x,y\in X,
\end{equation}
where $\alpha>0$ and $L\geq 0$ such that $r(\alpha+L)<1$. Then there exists a unique $x_0\in X$ which is both a startpoint and an endpoint of $J$ and $F$ if and only if $J$ and $F$ have the approximate mix-point property.
\end{theorem}

\section{Concluding remarks}

All the above results remain true if instead we consider a quasi-pseudometric type space $(X,d,b)$ (see \cite{eniola}). On the other side, the sets $C_n$ considered in the investigation can be made more general in the sense that we could consider the net $\{C_\varepsilon\}_{\varepsilon>0}$ where the $C_\varepsilon$ are sets of the form:
\[
C_\varepsilon= \left\lbrace x \in X: \underset{y\in Fx}{\sup} d^s(Jx,y)\leq \varepsilon \right\rbrace, \quad \text{ for any } \varepsilon>0.
\]


Therefore, the Theorem \ref{Res1} could be reformulated as follows:

\begin{theorem}\label{Resf1}
Let $(X,d,b)$ be a bicomplete quasi-pseudometric type space. Assume $J:X\to X$ is a continuous single-valued map such that $rd(x,y)\leq d(Jx,Jy)$ for some constant $r>0$ and $ \text{ for each } x,y\in X$. Let $F:X\to CB(X)$ be a set-valued map that satisfies
\begin{equation}
H(Fx,Fy) \leq \alpha d(Jx,Jy), \ \text{ \ for each } x,y\in X,
\end{equation}
where $\alpha\in (0,1)$ such that $r\alpha b^2<1$. Then there exists a unique $x_0\in X$ which is both a startpoint and an endpoint of $J$ and $F$ if and only if $J$ and $F$ have the approximate mix-point property.
\end{theorem}

In proving this Theorem \ref{Resf1}, the following lemma is key:

\begin{lemma}

Let $(X,d,b)$ be a bicomplete quasi-pseudometric type space. Assume $J:X\to X$ is a continuous single-valued map such that $rd(x,y)\leq d(Jx,Jy)$ for some constant $r>0$ such that $r\alpha b^2<1$ and $ \text{ for each } x,y\in X$. Let $F:X\to CB(X)$ be a set-valued map that satisfies
\begin{equation}
H(Fx,Fy) \leq \alpha d(Jx,Jy), \ \text{ \ for each } x,y\in X,
\end{equation}
where $\alpha\in (0,1)$. Then

\[
\delta(C_\varepsilon) \leq \frac{b\varepsilon(1+b)}{r(1-\alpha^2b)}, \quad \text{ for any } \varepsilon>0.
\]
\end{lemma}

\vspace*{1.5cm}

Similarly, the Theorem \ref{Res2} could be reformulated as follows:

\begin{theorem}\label{Resf2}
Let $(X,d,b)$ be a bicomplete quasi-pseudometric type space. Assume $J:X\to X$ is a continuous single-valued map such that $rd(x,y)\leq d(Jx,Jy)$ for some constant $r>0$ and $ \text{ for each } x,y\in X$. Let $F:X\to CB(X)$ be a set-valued map that satisfies
\begin{equation}
H(Fx,Fy) \leq \alpha [d(Jx,Fx)+d(Jy,Fy)], \ \text{ \ for each } x,y\in X,
\end{equation}
where $\alpha\in (0,1/2)$. Then there exists a unique $x_0\in X$ which is both a startpoint and an endpoint of $J$ and $F$ if and only if $J$ and $F$ have the approximate mix-point property.
\end{theorem}

The key lemma for Theorem \ref{Resf2} is

\begin{lemma}

Let $(X,d,b)$ be a bicomplete quasi-pseudometric type space. Assume $J:X\to X$ is a continuous single-valued map such that $rd(x,y)\leq d(Jx,Jy)$ for some constant $r>0$ and $ \text{ for each } x,y\in X$. Let $F:X\to CB(X)$ be a set-valued map that satisfies
\begin{equation}
H(Fx,Fy) \leq \alpha [d(Jx,Fx)+d(Jy,Fy)], \ \text{ \ for each } x,y\in X,
\end{equation}
where $\alpha\in (0,1/2)$. Then

\[
\delta(C_\varepsilon) \leq \frac{b\varepsilon}{r}(1+b+2\alpha b), \quad \text{ for any } \varepsilon>0.
\]
\end{lemma}

\vspace*{1.5cm}

In a similar manner, the Theorem \ref{Res3} could be reformulated as follows:

\begin{theorem}\label{Resf3}
Let $(X,d,b)$ be a bicomplete quasi-pseudometric type space. Assume $J:X\to X$ is a continuous single-valued map such that $rd(x,y)\leq d(Jx,Jy)$ for some constant $r>0$ and $ \text{ for each } x,y\in X$. Let $F:X\to CB(X)$ be a set-valued map that satisfies
\begin{equation}
H(Fx,Fy) \leq \alpha[ d(Jx,Fy ) + d(Fx,Jy)], \ \text{ \ for each } x,y\in X,
\end{equation}
where $0<\alpha<1/2$ with $2b^2r\alpha<1$. Then there exists a unique $x_0\in X$ which is both a startpoint and an endpoint of $J$ and $F$ if and only if $J$ and $F$ have the approximate mix-point property.
\end{theorem}

The key lemma for Theorem \ref{Resf3} is therefore:

\begin{lemma}
Let $(X,d,b)$ be a bicomplete quasi-pseudometric type space. Assume $J:X\to X$ is a continuous single-valued map such that $rd(x,y)\leq d(Jx,Jy)$ for some constant $r>0$ and $ \text{ for each } x,y\in X$. Let $F:X\to CB(X)$ be a set-valued map that satisfies
\begin{equation}
H(Fx,Fy) \leq \alpha[ d(Jx,Fy ) + d(Fx,Jy)], \ \text{ \ for each } x,y\in X,
\end{equation}
where $0<\alpha<1/2$ with $2b^2r\alpha<1$. Then

\[
\delta(C_\varepsilon) \leq \frac{b\varepsilon}{r(1-2b^2\alpha)}(1+b+2\alpha b), \quad \text{ for any } \varepsilon>0.
\]
\end{lemma}

Finally the Theorem \ref{Res4} could be reformulated as follows:

\begin{theorem}\label{Resf4}
Let $(X,d,b)$ be a bicomplete quasi-pseudometric type space. Assume $J:X\to X$ is a continuous single-valued map such that $rd(x,y)\leq d(Jx,Jy)$ for some constant $r>0$and $ \text{ for each } x,y\in X$. Let $F:X\to CB(X)$ be a set-valued map that satisfies
\begin{equation}
H(Fx,Fy) \leq \alpha d(Jx,Jy ) + L d(Fx,Jy), \ \text{ \ for each } x,y\in X,
\end{equation}
where $\alpha>0$ and $L\geq 0$ such that $rb^2(\alpha+bL)<1$. Then there exists a unique $x_0\in X$ which is both a startpoint and an endpoint of $J$ and $F$ if and only if $J$ and $F$ have the approximate mix-point property.
\end{theorem}

The proof of Theorem \ref{Resf4} will be done with the use of the following lemma:

\begin{lemma}

Let $(X,d,b)$ be a bicomplete quasi-pseudometric type space. Assume $J:X\to X$ is a continuous single-valued map such that $rd(x,y)\leq d(Jx,Jy)$ for some constant $r>0$ and $ \text{ for each } x,y\in X$. Let $F:X\to CB(X)$ be a set-valued map that satisfies
\begin{equation}
H(Fx,Fy) \leq \alpha d(Jx,Jy ) + L d(Fx,Jy), \ \text{ \ for each } x,y\in X,
\end{equation}
where $\alpha>0$ and $L\geq 0$ such that $rb^2(\alpha+bL)<1$. Then

\[
\delta(C_\varepsilon) \leq \frac{b\varepsilon(1+b+Lb^2)}{r(1-b^2(\alpha+bL))}, \quad \text{ for any } \varepsilon>0.
\]
\end{lemma}

{\bf Conflict of interest.}

The author declares that there is no conflict
of interests regarding the publication of this article.

\bibliographystyle{amsplain}

\end{document}